\documentclass[11pt]{article} % use larger type; default would be 10pt
\usepackage{etex}   %because using TikZ causes some counters to be exceeded; this increases the available registers from 256 to a bigger number.

   \usepackage{setspace} % to set line spaces (double, 1.5x. whatever)
   \usepackage{amsmath,mathtools}
   \usepackage{amsthm}
   \usepackage{amsfonts}   % if you want the fonts
   \usepackage{amssymb}

   \usepackage[all]{xy}   % for drawing diagrams
   \usepackage{tikz}{     % for drawing pictures
   \usetikzlibrary{calc}   % allows calculation of coordinates
   \usetikzlibrary{shapes.geometric} %for different shaped nodes

   \usepackage{color}
   \usepackage{soul}    %for highlighting your comments to yourself so that they stand out; use \hl{text}

	\usepackage{multicol}  % this lets you enumerate in several columns rather than just one

\usepackage{hyperref}
\hypersetup{
    colorlinks,
    citecolor=black,
    filecolor=black,
    linkcolor=black,
    urlcolor=black}  %turn the table of contents into a clickable one

\usepackage{enumerate}

\usepackage[utf8]{inputenc} % set input encoding (not needed with XeLaTeX)

\usepackage[top=3cm, bottom=3cm, left=2cm, right=2cm]{geometry}
\usepackage{graphicx} % support the \includegraphics command and options

\usepackage{booktabs} % for much better looking tables
\usepackage{array} % for better arrays (eg matrices) in maths
\usepackage{paralist} % very flexible & customisable lists (eg. enumerate/itemize, etc.)
\usepackage{verbatim} % adds environment for commenting out blocks of text & for better verbatim
\usepackage{subfig} % make it possible to include more than one captioned figure/table in a single float
\usepackage{fancyhdr} % This should be set AFTER setting up the page geometry
\pagestyle{fancy} % options: empty , plain , fancy
\lhead{}\chead{}\rhead{}
\lfoot{}\cfoot{\thepage}\rfoot{}

\usepackage{sectsty}
\allsectionsfont{\sffamily\mdseries\upshape} % (See the fntguide.pdf for font help)
\usepackage[nottoc,notlof,notlot]{tocbibind} % Put the bibliography in the ToC
\usepackage[titles,subfigure]{tocloft} % Alter the style of the Table of Contents

\newtheorem{theorem}{Theorem}[section]
\newtheorem{lemma}[theorem]{Lemma}
\newtheorem{proposition}[theorem]{Proposition}
\newtheorem{corollary}[theorem]{Corollary}
\newtheorem{remark}[theorem]{Remark}
\newtheorem{example}[theorem]{Example}
\theoremstyle{definition}
\newtheorem{definition}[theorem]{Definition}
\newtheorem{question}{Question}

\usepackage[all]{xy}

\newcommand{\GL}{\mathop{\mathrm{GL}}}
\newcommand{\SL}{\mathop{\mathrm{SL}}}
\newcommand{\Lie}{\mathop{\mathrm{Lie}}}
\newcommand{\End}{\mathop{\mathrm{End}}}
\newcommand{\Hom}{\mathop{\mathrm{Hom}}\nolimits}

\newcommand{\Ext}{\mathop{\mathrm{Ext}}\nolimits}        %the \nolimits thing is to make sure the indices don't go on top/bottom when displayed
\DeclareMathOperator{\Ima}{Im}

\title{Reductive pairs arising from representations}
\author{Oliver Goodbourn\thanks{The author gratefully acknowledges partial support from EPSRC grant EP/L001462/1}}
\date{} % Activate to display a given date or no date (if empty),
\begin{document}

\maketitle
\begin{abstract}
Let $G$ be a reductive algebraic group and $V$ a $G$-module.  We consider the question of when $(\GL(V),\rho(G))$ is a reductive pair of algebraic groups, where $\rho$ is the representation afforded by $V$.  We first make some observations about general $G$ and $V$, then specialise to the group $\SL_2(K)$ with $K$ algebraically closed of positive characteristic $p$.  For this group we provide complete answers for the classes of simple and Weyl modules, the behaviour being determined by the base $p$ expansion of the highest weight of the module.  We conclude by illustrating some of the results from the first section with examples for the group $\SL_3(K)$.
\end{abstract}

\section{Introduction}

Since their introduction in 1967 by Richardson \cite{MR0217079}, reductive pairs have featured in the proofs of numerous results in algebraic group theory \cite{MR2178661,MR2861529,MR2608407,MR3042602, MR1635690}.   For a definition, see the preliminaries section.   In a very loose sense, they are sometimes employed when seeking to prove results that attempt to salvage the good behaviour of groups over fields of characteristic $0$ in the positive characteristic case.  Often, the likelihood of such behaviour being correctly modelled increases as the characteristic gets larger.  In many cases this is reflected in the relative likelihood of a particular pair of algebraic groups being a reductive pair in larger characteristics.  For instance, when the characteristic is large compared to the dimension of a $G$-module $V$,  \cite[3.1]{MR2861529} tells us that $(\GL(V),\rho(G))$ is always a reductive pair, where $\rho$ is the representation afforded by $V$.  In the same paper, this fact is exploited to great effect in giving a conceptually uniform proof of a result giving necessary and sufficient conditions (for large enough characteristics) for subgroups of a connected reductive group $G$ to be $G$-completely reducible.  This came after work of Serre in \cite{MR2167207}, where a more favourable bound on the characteristic was achieved, but using a comparatively technical method.

In this paper we consider the question of which rational modules give rise to reductive pairs.  That is, in the situation above, \begin{question}\label{TheQuestion}for which $G$-modules $V$ is $(\GL(V), \rho(G))$ a reductive pair?\end{question}   Our approach will be largely representation theoretic:  we will examine the submodules of $V\otimes V^* \cong \Lie \GL(V)$.  However, subtleties in the nature and behaviour of $G$ as an algebraic group will arise, which must be dealt with before we may bring the full force of the representation theory to bear. 

We derive complete pictures for simple modules and Weyl modules for the group $\SL_2$ in positive characteristic.  The proof of the result for simple modules exploits Doty and Henke's tensor product decomposition for simple modules for this group.  This work is then used to determine the behaviour for Weyl modules,  making use of character calculations due to Donkin.  

We present several results and methods applicable to modules for arbitrary simple algebraic groups.  With these, it is straight-forward to generate large classes of examples (and non-examples) of modules giving reductive pairs; we give examples of this for the group $\SL_3$.

\section{Preliminaries}

For unexplained terminology, the reader may consult \cite{MR2015057}.  Throughout, let $K$ be an algebraically closed field of positive characteristic $p$.  Let $G$ be a reductive algebraic group over $K$, $T$ a maximal torus of $G$ and $B$ a Borel subgroup of $G$ containing $T$.  Let $\Phi$ be the root system of $G$ with respect to $T$, and let $\Psi$ be the set of simple roots determined by $B$.  We write $X(T)$ for the set of weights of $G$ with respect to $T$, and $X(T)_+$ for the set of dominant weights.   The choice of $T$ and $B$ determines a basis of fundamental dominant weights; we will often specify weights by their coordinates with respect to this basis.  We write $W$ for the Weyl group of $G$, and $W_p$ for the affine Weyl group; these will act on weights via the dot action.  

Unless otherwise stated, by module we will mean finite-dimensional rational $KG$-module.  Given $\lambda\in X(T)_+$, we write $\nabla(\lambda)$, $\Delta(\lambda)$ and $L(\lambda)$ for the induced, Weyl and simple modules with highest weight $\lambda$, respectively.  Given a module $V$, we write $V^*$ for its linear dual and $V_\lambda$ for the $\lambda$ weight space of $V$.  We denote by $V^{F^i}$  the $i$\textsuperscript{th} Frobenius twist of the module $V$; thus if $\rho$ is the representation afforded by $V$, then that afforded by $V^{F^i}$ is $\rho \circ F^i$.

Some of our methods will use the properties of tilting modules for algebraic groups.  A reader unfamiliar with these may consult \cite{MR1200163}; we will need only the following facts.  We will call a module \emph{tilting} if it has both a good filtration and a Weyl filtration.  The indecomposable tilting modules are classified by highest weights: for each dominant weight $\lambda$ there exists an indecomposable tilting module $T(\lambda)$ with unique highest weight $\lambda$, and the $T(\mu)$ with $\mu\in X(T)_+$ form a complete set of non-isomorphic indecomposable tilting modules.  Direct sums and tensor products of (now arbitrary) tilting modules are tilting, as are direct summands of tilting modules.  Crucially, two tilting modules are isomorphic if and only if they have the same formal character.

%Definition of a reductive pair
\begin{definition}
Let $H$ be a closed, reductive subgroup of a reductive group $G$.  We say that $(G,H)$ is a \emph{reductive pair} if $\Lie H$ is an $H$-module direct summand of $\Lie G$, where $H$ acts via the adjoint representation of $G$.
\end{definition}
In \cite{MR1635690}, Slodowy collects together many examples of reductive pairs.

\section{General statements}

We first recall two results of Benson and Carlson.  These may be found in \cite{MR866779}.

\begin{proposition}[Benson and Carlson]\label{BensonCarlsonTrivial}
Let $M$ and $N$ be finite dimensional, indecomposable $KG$-modules.  Then $K$ is a summand of $M\otimes N$ if and only if the following two conditions are met: \begin{enumerate}
		\item $M \cong N^*$
		\item $p \nmid \dim N$
	\end{enumerate} 
Moreover, if $K$ is a direct summand of $N\otimes N^*$ then it occurs with multiplicity one.
\end{proposition}

\begin{corollary}[Benson and Carlson]\label{BensonCarlson}
Suppose $M$ is an indecomposable $KG$-module with $p\mid\dim M$.  Then for any $KG$-module $N$ and any indecomposable summand $U$ of $M \otimes N$, we have $p\mid\dim U$.
\end{corollary}

Using the corollary it is easy to show the following constraint on when a module gives a reductive pair.

\begin{proposition}\label{DimensionConstraint}
Let $K$ have characteristic $p$, and suppose $p\nmid \dim \Lie G$.  Let $V=V_1\otimes \cdots \otimes V_r$ be a $G$-module such that one of the $V_i$ is indecomposable and has dimension divisible by $p$.  Then $V$ does not give a reductive pair.
\end{proposition}
\begin{proof}
Suppose the factor $V_i$ is indecomposable and has dimension divisible by $p$.  Write 
\begin{displaymath}
	V\otimes V^* \cong V_i \otimes (V_1\otimes \cdots \otimes V_r \otimes (V_1 \otimes \cdots \otimes V_r)^*),
\end{displaymath}
where we have rearranged the tensor product to bring $V_i$ to the front.  Now corollary~\ref{BensonCarlson} implies that all indecomposable summands of $V\otimes V^*$ have dimension divisible by $p$.  In particular, $\Lie G$ (being indecomposable) cannot be a summand, hence the result. 
\end{proof}

\begin{lemma}
\label{VAndVF} %V^F iff V  
Let $G$ be a reductive algebraic group over $K$ and $V$ a $G$-module.  Then $V^F$ gives a reductive pair if and only if $V$ does.
\end{lemma}
\begin{proof}
Recall that $V^F$ and $V$ are equal as $K$-vector spaces, so that $GL(V)=GL(V^F)$.  Let the representation afforded by the module $V$ be denoted by $\rho$.  Since the Frobenius morphism is a bijection, the subgroups $\rho(G)$ and $(\rho \circ F) (G)$ of $GL(V)$ are equal, from which the result follows. 
\end{proof}

\begin{remark}
For a simply connected semi-simple algebraic group, the above lemma is especially helpful when Steinberg's tensor product theorem is taken into account: if $\lambda = p\mu$, with $\lambda$, $\mu$ dominant weights, then $L(\lambda) = L(\mu)^F$.
\end{remark}

\begin{lemma}\label{GeneralExampleMachine}
Let $V,W$ be $G$-modules and let $\rho$ be the representation afforded by $V$. Suppose that $V$ gives a reductive pair and that the differential $d\rho$ is injective; suppose further that $\End(W)$ has a summand isomorphic to $K$.  Then the module $V\otimes W^F$ gives a reductive pair.
\end{lemma}
\begin{proof}
Let $\rho$ and $\sigma$ be the representations afforded by $V$ and $W$ respectively.  If we consider the representation \begin{displaymath}\rho\otimes\sigma^F:G\to\GL(V\otimes W^F),\end{displaymath} then the aim is to show that $\Lie(\rho\otimes\sigma^F(G))$ is a summand of $\End(V\otimes W^F)$.  First, we have the differential $d(\rho\otimes\sigma^F):\Lie G\to\End(V\otimes W^F)$, and \begin{equation}\label{ImageContained}\Ima(d(\rho\otimes\sigma^F))\subset \Lie(\rho\otimes\sigma^F(G)),\end{equation} that is, the image of the differential is contained in the Lie algebra of the image of the representation.

By the properties of the differential (\cite[3.21]{MR1102012}), $d(\rho\otimes\sigma^F) = d\rho\otimes 1_{W^F} + 1_V\otimes d\sigma^F$ (where here $1_{W^F}$ refers to the identity map on $W^F$); since the map $\sigma^F$ is equal to $\sigma \circ F$ and the differential of the Frobenius morphism $F$ is $0$, we have that $d\sigma^F=0$, whence $d(\rho\otimes\sigma^F) = d\rho \otimes 1_{W^F}$.  Thus \begin{displaymath}\dim d(\rho\otimes \sigma^F)(\Lie G) = \dim (d \rho \otimes 1_{W^F}) (\Lie G), \end{displaymath} which, since $d\rho$ is an isomorphism onto its image, is equal to $\dim \Lie G$.  However, \begin{displaymath}\dim \Lie G \geq \dim \Lie(\rho\otimes\sigma^F(G)),\end{displaymath} since the Lie algebra on the right is that of an algebraic group morphic image of the group the Lie algebra of which is on the left.  We therefore have equality in \ref{ImageContained}.  Thus we may look for the image of the differential rather than the Lie algebra of the image when deciding if $V\otimes W^F$ gives a reductive pair.

We note that $\End(V\otimes W^F) \cong \End(V)\otimes \End(W^F)$  \cite[Ch II, 4.4, prop. 4]{MR1727844}.  It will be convenient to identify these spaces via such an isomorphism.  Thus \begin{displaymath}d(\rho\otimes\sigma^F)=d\rho\otimes 1_{W^F}:\Lie G\to \End(V)\otimes \End(W^F).\end{displaymath} 
So we have $\Lie (\rho \otimes \sigma^F(G) )= (d\rho\otimes 1_{W^F})(\Lie G) = \{ d\rho X \otimes 1_{W^F} \mid X \in \Lie G\} = d\rho \Lie G \otimes K^F$.  Finally, we note that since $\Lie G$ is a summand of $\End(V)$ and $K^F$ is a summand of $\End(W^F)$, their tensor product $\Lie G \otimes K^F$ is a summand of $\End(V)\otimes \End(W^F) \cong \End(V\otimes W^F)$.
\end{proof}

\begin{corollary}\label{SimpleExampleMachine}
Let $G$ be a simple algebraic group over a field $K$ of positive characteristic $p$ such that $p$ is very good for $G$, let $\lambda$ be a restricted dominant weight such that the simple $G$-module $L(\lambda)$ gives a reductive pair, and let $\mu$ be a dominant weight such that $p \nmid \dim L(\mu)$.  Then the module $L(\lambda + p^n \mu)$ gives a reductive pair for any integer $n\geq 1$.
\end{corollary}
\begin{proof}
Since $K$ is infinite and $p$ is very good for the simple group $G$, the Lie algebra of $G$ is a simple $G$-module (see for instance \cite{MR2608407}). 
Thus any homomorphism leaving $\Lie G$ is either the zero map or is injective; since $\lambda$ is restricted, the differential is therefore injective.  
Since $L(\lambda)$ gives a reductive pair, we know that $\Lie G$ is a summand of $\Lie \GL(V)$.  Since $p\nmid \dim L(\mu)$, proposition~\ref{BensonCarlsonTrivial} tells us that $L(\mu)\otimes L(\mu)^*$ has a summand isomorphic to $K$.  We now apply lemma~\ref{GeneralExampleMachine}, noting that $L(\lambda)\otimes L(\mu)^{F^n} \cong L(\lambda + p^n\mu)$.
\end{proof}

\begin{remark}\label{GetMoreExamples}
If we relax the conditions on $\lambda$ and $\mu$, more may be said.  If $\lambda,\mu$ are required only to be dominant weights, then it may be that $L(\lambda)\otimes L(\mu)^{F^n}$ is not a simple module; however, the lemma still applies, so this module still gives a reductive pair.  Weaker conditions can be found to ensure that the module is still simple, for instance requiring that $\lambda \in X_n(T)$ with non-zero restricted part.
\end{remark}

\begin{lemma}\label{SplitSequenceInduction}
Let $R$ be a ring and let $M$ be an $R$-module with a filtration $$0 \leq M_1 \leq M_2 \leq \ldots \leq M_n =M. $$  Suppose every quotient $L_i :=  M_i / M_{i-1}$ is such that $\Ext^1_R(L_i, M_1) = 0$. (Equivalently, every short exact sequence $0 \to M_1 \to E \to L_i \to 0$ splits).  Then $M_1$ is an $R$-module direct summand of $M$: $M=M_1 \oplus W$ for some $W\leq M$.
\end{lemma}

\begin{proof}
 Given a short exact sequence $0\to A' \to A \to A'' \to 0$ of $R$-modules and an $R$-module $B$, we may consider the long exact sequence of the Ext functor in the first variable, $$\cdots \to \Ext_R^{n}(A'', B) \to \Ext_R^{n}(A,B) \to \Ext_R^{n}(A',B) \to \Ext_R^{n+1}(A'',B) \to \cdots . $$  In the notation of the statement, we are given exact sequences $0 \to M_{i-1} \to M_i \to L_i \to 0$, and we therefore have exact sequences $$\Ext_R^{1}(L_i,M_1)\to \Ext_R^{1}(M_i,M_1)\to\Ext_R^{1}(M_{i-1},M_1), $$ in which the first term is $0$ by hypothesis, and the last term is $0$ by induction (the base case being clear).  Thus $M_1$ is a summand of each $M_i$, in particular $M_n=M$.
\end{proof}

\begin{remark}\label{CompFacsMethod}
Provided that we can work out the composition factors of $V\otimes V^*$ for a $G$-module $V$, 
we may sometimes use lemma~\ref{SplitSequenceInduction} to show that $V$ gives a reductive pair (this method will not tell us that a given module \emph{does not} give a reductive pair).  We note that (the image of ) the Lie algebra of the image of $G$ is a simple submodule of $V\otimes V^*$ (subject, potentially, to minor constraints on the characteristic), so that $0\leq \Lie G \leq V\otimes V^*$; this may be refined into a composition series for $V\otimes V^*$ with $\Lie G$ at the bottom.  Thus, if we know a posteriori that all extensions of the Lie algebra by the other composition factors must be split, lemma~\ref{SplitSequenceInduction} tells us that the Lie algebra is a direct summand of $V\otimes V^*$ (whence the module $V$ gives a reductive pair).

We also note at this point that lemma~\ref{SplitSequenceInduction} implies equally that all the submodules of $V\otimes V^*$ in the same isomorphism class as $\Lie G$ are summands under the same hypotheses.  
\end{remark}

It will be useful to have some results that tell us about extensions of simple modules.  First of all, we recall the linkage principle (see for instance \cite[Corollary~6.17]{MR2015057}).
\begin{proposition}\label{LinkagePrinciple}
Let $\lambda,\mu\in X(T)_+$.  If $\Ext^1_G(L(\lambda),L(\mu))\neq 0$, then $\lambda\in W_p\cdot \mu$.
\end{proposition} 
We will use the contrapositive statement, that if $\lambda \not\in W_p\cdot\mu$, then any extension of $L(\lambda)$ by $L(\mu)$ must be split.

The following lemma is found in \cite[Section~2.12]{MR2015057}. %That's Jantzen
\begin{lemma}\label{ExtLambdaLambda}
The group $\Ext^1_G(L(\lambda),L(\lambda))=0$ for all $\lambda \in X(T)_+$.
\end{lemma}

\section{The group $\SL_2$}

\subsection{Basic facts}
In this section we focus on the group $G=\SL_2(K)$.  A great deal is known about this group and its representations, making it a logical choice for a first example.  For dominant weights $0\leq u \leq p-1$ we have $T(u)=L(u)=\nabla(u)=\Delta(u)$.  For $p\leq u \leq 2p-2$ the module $T(u)$ is uniserial with composition factors $L(2p-2-u),L(u),$ and $L(2p-2-u)$.  It is well-known (see for instance \cite{MR2143497}) that for $\SL_2(K)$ the module $\nabla(r) \cong S^rE$, the $r^{\mathrm{th}}$ symmetric power of the $2$-dimensional natural module for $\SL_2(K)$.  Thus the simple $\SL_2(K)$-modules are tensor products of Frobenius twists of such symmetric powers, by Steinberg's tensor product theorem.  The simple $\SL_2(K)$-modules are self-dual. As in \cite{MR2143497}, we shall call the indecomposable tilting modules $T(u)$ with $0\leq u \leq 2p-2$ \emph{fundamental}.

\subsection{Simple Modules}

In this section we show necessary and sufficient conditions for a simple module for $\SL_2(K)$ to give a reductive pair.  

\begin{lemma}\label{LookForL2}
Let $K$ have characteristic $p\geq 3$, and let $\rho: \SL_2(K) \to \GL(V)$ be a nontrivial rational representation.  Then $\Lie \rho( \SL_2(K))$ is isomorphic to the Lie algebra $\mathfrak{sl}_2(K)$ of trace-zero $2\times 2$ matrices with entries in $K$.
\end{lemma}
\begin{proof}
The kernel of $\rho$ is a closed normal subgroup of $\SL_2(K)$ and is therefore either trivial or is the centre of $\SL_2(K)$, which consists of those scalar matrices having determinant $1$.  The image $\rho(\SL_2(K))$ is thus isomorphic as an abstract group to $\SL_2(K)$, and as an algebraic group either to $\SL_2(K)$ or to $\mathrm{PGL}_2(K)$.  Since $p\nmid 2$, the Lie algebras of both of these groups are isomorphic, and in particular are isomorphic to $\mathfrak{sl}_2(K)$.
\end{proof}

It is well known that if $K$ has characteristic $p$ and $p\nmid n$, then $\Lie \SL_n(K)$ is an irreducible $\SL_n(K)$-module.  Thus when deciding whether or not a representation of $\SL_2(K)$ gives a reductive pair, if the characteristic is greater than or equal to $3$, we may be certain the Lie algebra we look for is isomorphic to $L(2)$, the $3$-dimensional simple $\SL_2$-module with highest weight $2$.  In \cite{MR3042602}, Herpel introduces the notion of pretty good primes.  In the paper he shows that when the characteristic of $K$ is not pretty good for $G$, there can be no reductive pair of the form $(\GL(V),G)$.  Since $2$ is not a pretty good prime for a group of type $A_1$, and since for a non-trivial representation the image of $\SL_2(K)$ in $\GL(V)$ is of this type, it follows that in characteristic $2$, no $\SL_2(K)$-modules give reductive pairs.

It is easy to calculate the dimensions of the simple modules for $SL_2$.  We may also infer an initial constraint on the simple modules that can possibly give reductive pairs.  Both observations are direct applications of Steinberg's tensor product theorem.
\begin{lemma}\label{SimpleModuleDimension}
Let $\lambda = a_0 + a_1p + \dots + a_rp^r$ be a non-negative integer.  Then the dimension of the simple $KSL_2$-module $L(\lambda)$ is $\prod_{i=0}^r(a_i+1)$. 
\end{lemma}

\begin{lemma}
Suppose $p>3$, and let $\lambda$ be a non-negative integer with base $p$ expansion $\lambda = a_0 + a_1p+\dots + a_kp^k$.  Supppose at least one of the $a_i=p-1$.  Then $L(\lambda)$ does not give a reductive pair.
\end{lemma}
\begin{proof}
 This follows upon application of proposition~\ref{DimensionConstraint}.
\end{proof}

It is a standard result that the module $\nabla(\mu)\otimes \nabla(\nu)$ has a good filtration with sections isomorphic to $\nabla(\mu + \nu), \nabla(\mu + \nu - 2), \ldots, \nabla(\mu - \nu)$ (each with multiplicity one, with $\nu \leq \mu$).

\begin{lemma}
\label{NablaAndDeltaHomSpaces}
$\dim \Hom_{SL_2}(\Delta(\lambda),\nabla(\mu)\otimes\nabla(\nu))\leq1$.
\end{lemma}
\begin{proof}
By the zero case of  \cite[Prop.~4.13]{MR2015057}, \begin{equation}\label{JantzenHomSpaces}\dim \Hom_G (\Delta(\alpha),\nabla(\beta))=\begin{cases} 1 &\mbox{if } \alpha=\beta \\ 0 &\mbox{if } \alpha\neq\beta. \end{cases}  \end{equation}
For a given module $X$ and a short exact sequence $0 \to Y_1 \to Y \to Y_2 \to 0$ of $KG$-modules, we have $\dim \Hom_G(X,Y) \leq \dim\Hom_G(X,Y_1)+\dim\Hom_G(X,Y_2)$.  We apply this result inductively to a good filtration of $\nabla(\mu)\otimes\nabla(\nu)$, giving \begin{displaymath} \dim \Hom_G(\Delta(\lambda),\nabla(\mu)\otimes \nabla(\nu)) \leq \sum_{i=0}^{\nu} \dim \Hom_G(\Delta(\lambda),\nabla(\mu + \nu -2i)). \end{displaymath}  Since at most one of the integers $\mu + \nu - 2i$ is equal to $\lambda$, formula~\ref{JantzenHomSpaces} implies that the right hand side is at most $1$.
\end{proof}

\begin{corollary}
\label{SimpleHomSpaces}
$\dim \Hom_{SL_2} (L(\lambda),L(\mu)\otimes L(\nu)) \leq 1$.  
\end{corollary}

\begin{proof}
Note that $\nabla(\lambda)$ has socle isomorphic to $L(\lambda)$, and $\Delta(\lambda)$ has head isomorphic to $L(\lambda)$.  Thus $\Hom_{SL_2}(L(\lambda),L(\mu)\otimes L(\nu))$ embeds in $\Hom_{SL_2}(\Delta(\lambda),L(\mu)\otimes L(\nu))$.   
Next, $\Hom_{SL_2}(\Delta(\lambda),L(\mu)\otimes L(\nu)) \leq \Hom_{SL_2}(\Delta(\lambda),\nabla(\mu)\otimes \nabla(\nu))$.  Thus it is enough to prove that this latter space has dimension $\leq 1$, which is lemma~\ref{NablaAndDeltaHomSpaces}.%For this one, do the map X to Z followed by inclusion Z into Y.  
\end{proof}

Combined with lemma~\ref{LookForL2}, the argument above shows in particular that in characteristic $p\neq2$ there is at most (hence exactly) one submodule of any tensor product $L(\lambda)\otimes L(\lambda)^* \cong L(\lambda)\otimes L(\lambda)$ that is isomorphic to $L(2)$, hence to the Lie algebra of the image of $SL_2$ in the representation $\rho$.  Thus, if we find such a submodule (whether or not it appears as a summand) we may be certain that it is in fact the Lie algebra of the image; this is a great aid in deciding whether an $SL_2$-module gives a reductive pair.

In \cite{MR2143497}, Doty and Henke provide a decomposition of the tensor product $L(a)\otimes L(b)$ of two simple modules for $\SL_2$ into indecomposable direct summands.  Each of the summands is a tensor product of Frobenius twists of fundamental tilting modules for $\SL_2$.

We are now ready to state the main result of this section, which classifies the simple $SL_2$-modules giving reductive pairs.  The majority of the work has been done already, and the proof will therefore take the form of a few observations, minor calculations and checks.

\begin{theorem}\label{simple}
Let $\lambda \neq 0 $ be a non-negative integer with base $p$ expansion $\lambda= a_0 + a_1p + \cdots + a_kp^k$ and $\rho:SL_2(K) \to GL(L(\lambda))$ the representation afforded by the simple module $L(\lambda)$.  Let $l$ be the smallest integer for which $a_l\neq 0$.  Then, when $p>3$, $(GL(L(\lambda)),\rho(SL_2))$ is a reductive pair if and only if \begin{enumerate}
\item $a_i \leq p-2$ for all $i$, and
\item $a_l \leq p-3$.                              
\end{enumerate}

If $p=3$, then $(GL(L(\lambda)),\rho(SL_2))$ is a reductive pair if and only if all $a_i \leq 1$ \emph{except for} $a_l$, which can be $1$ \emph{or} $2$.

If $p=2$, $(GL(L(\lambda)),\rho(SL_2))$ is never a reductive pair.
\end{theorem}
\begin{proof}
Considering lemma~\ref{VAndVF}, we may assume $l=0$.  Applying \cite[Theorem 2.1]{MR2143497}, we see that the indecomposable summands of $L(\lambda)\otimes L(\lambda)^*$ all take the form $\bigotimes_i T(u_i)^{F^i}$, where the $u_i$ are determined by an easy combinatorial process and depend only on the base $p$ expansion of $\lambda$.  By lemma~\ref{SimpleHomSpaces} (and Krull-Schmidt), it is therefore enough to determine whether or not one of these summands is isomorphic to the Lie algebra of the image of $\SL_2$.  

In characteristic $p\geq 3$, this module is isomorphic to $L(2)$, and one calculates that there is a summand $\bigotimes_i T(u_i)^{F^i}$ isomorphic to $L(2)$ precisely when the stated conditions are met.  It was already noted above that no $\SL_2(K)$ module gives a reductive pair when $p=2$.  For simple modules, this can also be seen directly in the representation theory, as \cite[Theorem 2.3]{MR2143497} implies that $L(\lambda)\otimes L(\lambda)^*$ is itself indecomposable.  Thus, to be a summand, $\Lie \rho(\SL_2)$ would need to be isomorphic to $L(\lambda)\otimes L(\lambda)^*$; this is impossible, since $\dim \Lie \rho(\SL_2) = 3$ is not a square integer.
\end{proof}

\subsection{Weyl Modules}

Define \begin{displaymath}Y(r):=\begin{cases} \nabla(m)\otimes\Delta(m) & \mbox{ if } r=2m \mbox{ is even,} \\ \nabla(m+1)\otimes \Delta(m) & \mbox{ if } r=2m+1 \mbox{ is odd.}\end{cases}\end{displaymath} Note that for even $r$, $Y(r)\cong \nabla(m)\otimes(\nabla(m))^*$.

Recall the definition of tilting modules from the preliminaries section.
The following lemma is a special case of \cite[Lemma 3.3]{AlisonHigherExtensions}, and is due to Donkin.
\begin{lemma}
\label{NablaDeltaTilting}
If $r,s\geq 0$ with $|r-s|\leq 1$, then $\nabla(r)\otimes \Delta(s)$ is a tilting module. 
\end{lemma}

It can be shown that for all non-negative, even integers $r$, the character of $Y(r)$ is a sum of characters of certain tilting modules.  This implies an isomorphism of tilting modules between $Y(r)$ and the direct sum of the tilting modules whose characters appear, and will allow us to reduce the problem of determining (for each non-negative integer $m$) whether $\nabla(m)$ gives a reductive pair for $SL_2$, to the already covered case of simple modules.  The character calculations involved were kindly shown to us by Stephen Donkin.
Let $p>2$ and let $0\leq a \leq p-3$ be an even integer.  We have

\begin{multline}\label{lastbit}Y(2pm+a) \cong \nabla(p-1)\otimes\nabla(a+1)\otimes Y(2m-1)^F \\ \oplus \nabla(p-1)\otimes\nabla(p-a-3)\otimes Y(2m-2)^F \oplus Y(a).\end{multline}

\begin{multline}\label{lastbit3}Y(p-1+2mp+a)\cong \nabla(p-1)\otimes\nabla(a)\otimes Y(2m)^F \oplus \nabla(p-1)\otimes\nabla(p-2-a)\otimes Y(2m-1)^F \\ \oplus Y(p-a-3).\end{multline}

\begin{align}\label{lastbit4}Y(2pm+2p-2)=\nabla(p-1)\otimes\nabla(p-1)\otimes Y(2m)^F.\end{align}

Considering \ref{lastbit}, \ref{lastbit3} and \ref{lastbit4}, we have expressions for $Y(r)$ for each even, non-negative $r$.  Thus we have expressions for $\nabla(m)\otimes\Delta(m)$ with $r=2m$ -- that is, all non-negative integers $m$.

\begin{theorem}\label{symmetric}
Let $K$ be an algebraically closed field of characteristic $p>0$, let $n$ be a non-negative integer and let $\rho:SL_2(K)\to GL(\nabla(n))$ be the representation afforded by $\nabla(n)$, the $n$\textsuperscript{th} symmetric power of the natural module.
\begin{enumerate}

\item If $K$ has characteristic $p>3$, then $(GL(\nabla(n)),\rho(SL_2))$ is a reductive pair if and only if $n\not\equiv p, p-1$ or $p-2$ (mod $p$).

\item If $K$ has characteristic $3$, then $(\GL(\nabla(n)), \rho(SL_2))$ is a reductive pair if and only if $n \equiv 1, 2, \dots 6$ (mod $9$).

\item If $K$ has characteristic $2$, then $(\GL(\nabla(n)), \rho(SL_2))$ is never a reductive pair.

\end{enumerate}
\end{theorem}

\begin{proof}

First suppose $K$ has characteristic $p>3$.  Although in proving this result we will deal with several cases, they all resolve to the same issue: what ultimately matters is the residue class of $n$ modulo $p$.  As in the calculations above, we write $r=2n$ and $\nabla(n)\otimes \nabla(n)^* \cong \nabla(n)\otimes \Delta(n) = Y(r)$ as a direct sum of tilting modules.  Recall that by \ref{NablaAndDeltaHomSpaces}, if we exhibit a summand of $\nabla(n)\otimes\Delta(n)$ that is isomorphic to $L(2)$, then we know $\nabla(n)$ does give a reductive pair.  On the other hand, if we can show that in a given decomposition of ${\nabla(n)\otimes\Delta(n)}$ into (not necessarily indecomposable) direct summands $M_i$, none of the $M_i$ has a summand isomorphic to $L(2)$, then $\nabla(n)$ cannot give a reductive pair, by the Krull-Schmidt theorem.

By equation~\ref{lastbit}, we see that when $n\equiv 0 , 1, \ldots , \frac{p-3}{2}$ (mod $p$), $\nabla(n)$ gives a reductive pair if and only if $L(2\overline{n})$ gives a reductive pair, where $\overline{n}$ is the least residue of $n$ modulo $p$.  When $\overline{n}=0$, $\nabla(n)$ does not give a reductive pair: by corollary~\ref{BensonCarlson}, we know that no module isomorphic to $L(2)$ can be a summand of any module of the form $\nabla(p-1)\otimes N$ (where $N$ is a $G$-module), since $\nabla(p-1)$ is $p$-dimensional and indecomposable (it is irreducible).  We combine this this with the last observation in the previous paragraph, noting that none of the summands in \ref{lastbit} has a summand isomorphic to $L(2)$.  In the rest of the cases we have $0<2\overline{n}\leq p-3$, so that $\nabla(n)$ does give a reductive pair.

By equation~\ref{lastbit3}, we see that if $n\equiv \frac{p-1}{2}, \frac{p-1}{2}+1, \ldots, p-3$ (mod $p$), then $\nabla(n)$ does give a reductive pair, while for $n\equiv p-2$ it does not.  This is because if $n\equiv \frac{p-1}{2} + \frac{a}{2}$ (mod $p$) where $0\leq a \leq p-3$, $a$ even, then $\nabla(n)$ gives a reductive pair if and only if $L(\frac{p-a-3}{2})$ gives a reductive pair, which is true for $0\leq a \leq p-5$ but false for $a=p-3$.

Finally we look at equation~\ref{lastbit4}.  Thus if $n\equiv p-1$ (mod $p$), then $\nabla(n)$ does not give a reductive pair, as corollary~\ref{BensonCarlson} shows $\nabla(n)\otimes\Delta(n)$ does not have a summand isomorphic to $L(2)$ (as before).  Combining this and the observations about other residue classes above, the proof for $p>3$ is complete.

Now suppose $p=3$.  In this case, letting $a=0$ in equation~\ref{lastbit}, we have 
\begin{equation}\label{lastbitp3}
Y(6m)\cong \nabla(2)\otimes\nabla(1)\otimes Y(2m-1)^F \oplus \nabla(2)\otimes\nabla(0)\otimes Y(2m-2)^F \oplus Y(0).
\end{equation}
Equation~\ref{lastbit3} becomes
\begin{equation}\label{lastbit2p3}
Y(6m+2)\cong \nabla(2)\otimes Y(2m)^F \oplus \nabla(2)\otimes\nabla(1)\otimes Y(2m-1)^F \oplus Y(0).\end{equation}
Finally, equation~\ref{lastbit4} becomes
\begin{equation}\label{lastbit4p3}
Y(6m+4)=\nabla(2)\otimes\nabla(2)\otimes Y(2m)^F.
\end{equation}
From \ref{lastbitp3}, we see that if $p\nmid m$, then $Y(6m)$ has a summand isomorphic to $L(2)$.  This is because the term $\left(\nabla(2)\otimes\nabla(0)\otimes Y(2m-2)^F \right)$ itself has a summand $L(2)\otimes L(0) \otimes L(0)^F$, using proposition~\ref{BensonCarlsonTrivial}, noting that $Y(2m-2)^F=\left(\nabla(m-1)\otimes\Delta(m-1)\right)^F$.  Since $Y(6m)=\nabla(3m)\otimes\Delta(3m)$, we therefore have that $\nabla(3m)$ gives a reductive pair for each $m$ coprime to $3$, namely ${\nabla(3\times 1),} \nabla(3\times 4), \nabla(3\times 7)$, or generally those of the form $\nabla(3+9k)$; and also $\nabla(3\times 2), \nabla(3\times 5)$, or in other words those of the form $\nabla(6+9k)$.  On the other hand it may be shown that $L(2)\otimes L(1)\cong T(3)$ in characteristic $3$.  Thus $L(2)$ cannot be a summand of either of the other terms in equation~\ref{lastbitp3}: the composition factors of $T(3)$ in characteristic $3$ are $L(1), L(3), L(1)$, as noted above.  Hence the composition factors of $T(3)\otimes Y(2m-1)^F$ are the composition factors of $L(1)\otimes Y(2m-1)^F$ (twice each) and the composition factors of $L(3)\otimes Y(2m-1)^F$. If $T(3)\otimes Y(2m-1)^F$ has a summand isomorphic to $L(2)$, then this summand is in particular a submodule, and, being simple, is therefore a composition factor.  Thus it is enough to know that $L(2)$ cannot be a composition factor of either $L(1)\otimes Y(2m-1)^F$ or $L(3)\otimes Y(2m-1)^F \cong (L(1)\otimes Y(2m-1))^F$; considering the weights of these modules, we see that no composition factor of either may have highest weight congruent to $2$ modulo $3$.  Finally, $Y(0)$ is $1$-dimensional, so $L(2)$ cannot be a submodule.

From \ref{lastbit2p3}, once again using proposition~\ref{BensonCarlsonTrivial}, $Y(6m+2)$ does have a summand isomorphic to $L(2)$ if $p\nmid m+1$, noting that $Y(2m)^F=\left(\nabla(m)\otimes\Delta(m)\right)^F$.  Then, following the same process as for the previous case, we see that we get a reductive pair from each $\nabla(1+9k)$ and each $\nabla(4+9k)$.  Since the other terms in equation~\ref{lastbit2p3} are the same as in the previous case, the same reasoning shows that $L(2)$ is not a summand of either of the those terms.

From \ref{lastbit4p3}, we see that $Y(6m+4)$ has a summand isomorphic to $L(2)$ if $p\nmid m+1$.  To see this, note that $L(2)\otimes L(2)\cong T(4) \oplus L(2)$ then apply the same reasoning as before.  In this case, we see that we get reductive pairs from $\nabla(2+9k)$ and $\nabla(5+9k)$.  Again, these are the only ways to get a summand isomorphic to $L(2)$.

To summarise: from the cases above, we see that when $p=3$, $\nabla(n)$ gives a reductive pair if and only if $n \equiv 1, 2, \dots 6$ (mod $9$).

When $p=2$, no $\SL_2(K)$-module gives a reductive pair.  
\end{proof}

\begin{remark}
 As with the case for simple modules above, it is also fairly easy to see directly from the representation theory that no Weyl module for $\SL_2(K)$ gives a reductive pair in characteristic $2$.  First recall that direct summands of tilting modules are tilting modules.  Thus, if the image of $\Lie \rho(\SL_2(K))$ in $\nabla(n)\otimes \Delta(n)$ is not a tilting module, it cannot be a summand.  With this in mind, consider that the differential $d\rho : \mathfrak{sl}_2(K) \to \Lie \left( \rho \SL_2(K)\right)$ is injective unless $n$ is even, in which case it has as its kernel the scalar matrices.  If $n$ is odd, $\Lie \rho(\SL_2)$ is therefore the $\SL_2(K)$-module $\nabla(2)$, which is not a tilting module.  Thus $\nabla(n)$ does not give a reductive pair when $n$ is odd in characteristic $2$.  If $n$ is even, then $\Lie \rho(\SL_2)$ has a $2$-dimensional simple submodule coming from the image $d\rho \left(\mathfrak{sl}_2(K)\right)$.  Since $\mathfrak{sl}_2(K)$ is the $\SL_2(K)$-module $\nabla(2)$, this simple module must be $L(2)$ (which in characteristic $2$ is $L(1)^F$, which is $2$-dimensional).  There are then two possibilities: either $\Lie \rho(\SL_2)$ is indecomposable, in which case it is the module $\Delta(2)$; or else it has a decomposition as $L(2)\oplus L(0)$.  Since $\Delta(2)$ and $L(2)$ are not tilting, in either of these cases this is enough information to conclude that $\nabla(n)$ does not give a reductive pair for even $n$ in characteristic $2$.
\end{remark}

\begin{example}

Let $k$ have characteristic $p>3$. By \ref{symmetric}, we see that $\nabla(p)$ does not give rise to a reductive pair; by \ref{simple}, we see that the simple module $L(1+\frac{p(p-1)}{2})= L(1) \otimes L(\frac{p-1}{2})^F$ does.  Both of these modules have dimension $p+1$, and we note that $2<p<2(p+1)-2 = 2p$.  Thus, if the characteristic is not $3$, we have examples of both sorts of behaviour between the bounds in \cite[3.1]{MR2608407}.
\end{example}

\section{The group $\SL_3$}

Now let $G$ = $\SL_3(K)$. Consider a rational representation $\rho: \SL_3(K)\to \GL(V)$.  If the representation is non-trivial, the kernel of $\rho$ is either trivial or the centre of $\SL_3(K)$, and is in either case finite.  When $p\neq 3$, the adjoint representation is irreducible and has highest weight $(1,1)$.  Although some progress has been made (e.g. \cite{MR2753767,1111.5811}), no complete tensor product decomposition (as in \cite{MR2143497}) seems to have appeared at the time of writing.  We will instead exploit the method set out in remark~\ref{CompFacsMethod} to generate some examples.

In Yehia's PhD thesis~\cite{Yehia}, $\Ext^1_G(L(\mu),L(\lambda))$ is shown to be at most one-dimensional for $G$ of type $A_2$.  Furthermore, for each dominant weight $\lambda$ the set of weights \begin{displaymath}A(\lambda):= \{\mu \in X(T)_+ \mid \Ext^1_G(L(\mu),L(\lambda)) \neq 0\}\end{displaymath} is determined explicitly.  For our purposes it will be enough to consider the following result, which is an abridgement of \cite[Proposition~4.1.1]{Yehia}.

\begin{proposition}\label{YehiaExtResult}
Suppose $\lambda, \mu\in X(T)_+$, $\lambda=\lambda_0 + p\lambda'$, $\mu = \mu_0 + p\mu'$, where $\lambda_0, \mu_0\in X_1$ (the restricted region) and $\lambda',\mu'\in X(T)_+$.  Moreover suppose $\lambda_0 \neq \mu_0$.  Then
  \begin{enumerate}
  \item A necessary condition for $\Ext_G^1(L(\mu),L(\lambda))$ to be non-zero is that $\mu_0 $ is in the $W_p$ orbit of $\lambda_0$.
  \item If $\lambda_0=(r,s)\in A_0$ (the bottom alcove) and $\mu_0$ is \emph{not} one of $(p-s-2,p-r-2), (r+s+1,p-s-2)$ or $(p-r-2,r+s+1)$, we have $\Ext^1_G(L(\mu),L(\lambda))=0$.
  \end{enumerate}
\end{proposition}

Hence for $\SL_3(K)$ we may do significantly better than using linkage alone.  The full result in \cite{Yehia} considers in point $(2)$ any $\lambda_0$ in the restricted region.  We may (somewhat crudely) use proposition~\ref{YehiaExtResult} in combination with the linkage principle~\ref{LinkagePrinciple} in order to create a ``mask" of those dominant weights for which extensions of the simple module $L(1,1)$ by simple modules with this highest weight are all split.  The result is a reduced selection of weights to look for amongst the composition factors of $V\otimes V^*$.  The following example illustrates the use of the method outlined in remark~\ref{CompFacsMethod} in this case.

\begin{example}
  Let $p=5$.  By considering characters, the composition factors of $L(5,1)\otimes L(5,1)^*$ are calculated to have highest weights (with multiplicities following) $$(6,6),1,(5,5),1,(1,1),1,(0,0)1.$$  The weight $(6,6)$ is in the $W_p$ orbit of $(1,1)$, so the linkage principle alone does not allow us to rule out this weight as possibly contributing a non-split extension.  However, $(6,6)$ is \emph{not} one of the weights specified by \ref{YehiaExtResult}.  We may therefore infer that $L(5,1)$ gives a reductive pair.
\end{example}

Stephen Doty has written a package~\cite{DotyWeylMods} for the computer algebra software GAP~\cite{GAP4} which can perform calculations pertaining to Weyl modules.  We used this software to calculate the composition factors of tensor products of simple modules with small highest weights in type $A_2$ in a variety of small characteristics.  As a further example, we combine information gathered using this GAP package with corollary~\ref{SimpleExampleMachine} to determine some classes of modules giving reductive pairs.
\begin{proposition}
Let $K$ have characteristic $7$ and let $n\geq 1$ be an integer.  Let $\lambda\in A=\{(1,0),(0,1),(1,1),\\(2,0),(0,2),(2,1),(1,2),(3,0),(0,3)\}$ and $\mu\in B=\{(1,0), (0,1), (1,1), (2,0),(0,2),(2,1),(1,2),(2,2),\\(3,0),(0,3),(3,1),(1,3)\}$.  Then the module $L(\lambda + p^n\mu)$ gives a reductive pair. 
\end{proposition}
\begin{proof}
The weights in the set $A$ of the statement are all restricted, and the simple modules with these highest weights give reductive pairs: this may be seen by examining the composition factors of the tensor product of a given simple module with its dual (obtained using Doty's GAP package), then applying the method in remark~\ref{CompFacsMethod}.  Since the weights in $B$ are all in the bottom alcove, \cite[II,5.6]{MR2015057} implies that each simple module with highest weight in $B$ is equal to the induced module with the same highest weight.  Then note that $p=7$ does not divide the dimensions (calculated using Weyl's dimension formula \cite[Cor.24.6]{FultonHarris}) of the these simple modules.  Thus the conditions of corollary~\ref{SimpleExampleMachine} are satisfied, and for $\lambda \in A, \mu \in B$ we have that $L(\lambda)\otimes L(\mu)^F \cong L(\lambda + p^n\mu)$ gives a reductive pair.
\end{proof}

\bibliographystyle{plain}
\bibliography{bibliography}

\def\cprime{$'$}
\begin{thebibliography}{10}

\bibitem{MR2861529}
M.~Bate, S.~Herpel, B.~Martin, and G.~R{\"o}hrle.
\newblock {$G$}-complete reducibility and semisimple modules.
\newblock {\em Bull. Lond. Math. Soc.}, 43(6):1069--1078, 2011.

\bibitem{MR2178661}
M.~Bate, B.~Martin, and G.~R{\"o}hrle.
\newblock A geometric approach to complete reducibility.
\newblock {\em Invent. Math.}, 161(1):177--218, 2005.

\bibitem{MR2608407}
M.~Bate, B.~Martin, G.~R{\"o}hrle, and R.~Tange.
\newblock Complete reducibility and separability.
\newblock {\em Trans. Amer. Math. Soc.}, 362(8):4283--4311, 2010.

\bibitem{MR866779}
D.~J. Benson and J.~F. Carlson.
\newblock Nilpotent elements in the {G}reen ring.
\newblock {\em J. Algebra}, 104(2):329--350, 1986.

\bibitem{MR1102012}
A.~Borel.
\newblock {\em Linear algebraic groups}, volume 126 of {\em Graduate Texts in
  Mathematics}.
\newblock Springer-Verlag, New York, second edition, 1991.

\bibitem{MR1727844}
N.~Bourbaki.
\newblock {\em Algebra {I}. {C}hapters 1--3}.
\newblock Elements of Mathematics (Berlin). Springer-Verlag, Berlin, 1998.
\newblock Translated from the French, Reprint of the 1989 English translation [
  MR0979982 (90d:00002)].

\bibitem{MR2753767}
C.~Bowman, S.~R. Doty, and S.~Martin.
\newblock Decomposition of tensor products of modular irreducible
  representations for {${\rm SL}_3$}.
\newblock {\em Int. Electron. J. Algebra}, 9:177--219, 2011.
\newblock With an appendix by C. M. Ringel.

\bibitem{1111.5811}
C.~Bowman, S.~R. Doty, and S.~Martin.
\newblock Decomposition of tensor products of modular irreducible
  representations for {${\rm SL}_3$}: the $p \geq 5$ case.
\newblock 2011.
\newblock arXiv:1111.5811.

\bibitem{MR1200163}
S.~Donkin.
\newblock On tilting modules for algebraic groups.
\newblock {\em Math. Z.}, 212(1):39--60, 1993.

\bibitem{DotyWeylMods}
S.~Doty.
\newblock Weyl modules gap package.
\newblock \url{http://doty.math.luc.edu/weylmodules}.
\newblock Accessed: 15/09/2014.

\bibitem{MR2143497}
S.~Doty and A.~Henke.
\newblock Decomposition of tensor products of modular irreducibles for {${\rm
  SL}_2$}.
\newblock {\em Q. J. Math.}, 56(2):189--207, 2005.

\bibitem{FultonHarris}
W.~Fulton and J.~Harris.
\newblock {\em Representation theory}, volume 129 of {\em Graduate Texts in
  Mathematics}.
\newblock Springer-Verlag, New York, 1991.
\newblock A first course, Readings in Mathematics.

\bibitem{GAP4}
The GAP~Group.
\newblock {\em {GAP -- Groups, Algorithms, and Programming, Version 4.7.5}},
  2014.

\bibitem{MR3042602}
Sebastian Herpel.
\newblock On the smoothness of centralizers in reductive groups.
\newblock {\em Trans. Amer. Math. Soc.}, 365(7):3753--3774, 2013.

\bibitem{MR2015057}
J.~C. Jantzen.
\newblock {\em Representations of algebraic groups}, volume 107 of {\em
  Mathematical Surveys and Monographs}.
\newblock American Mathematical Society, Providence, RI, second edition, 2003.

\bibitem{AlisonHigherExtensions}
A.~E. Parker.
\newblock Higher extensions between modules for {$\rm SL_2$}.
\newblock {\em Adv. Math.}, 209(1):381--405, 2007.

\bibitem{MR0217079}
R.~W. Richardson, Jr.
\newblock Conjugacy classes in {L}ie algebras and algebraic groups.
\newblock {\em Ann. of Math. (2)}, 86:1--15, 1967.

\bibitem{MR2167207}
Jean-Pierre Serre.
\newblock Compl\`ete r\'eductibilit\'e.
\newblock {\em Ast\'erisque}, (299):Exp. No. 932, viii, 195--217, 2005.
\newblock S{\'e}minaire Bourbaki. Vol. 2003/2004.

\bibitem{MR1635690}
Peter Slodowy.
\newblock Two notes on a finiteness problem in the representation theory of
  finite groups.
\newblock In {\em Algebraic groups and {L}ie groups}, volume~9 of {\em Austral.
  Math. Soc. Lect. Ser.}, pages 331--348. Cambridge Univ. Press, Cambridge,
  1997.
\newblock With an appendix by G.-Martin Cram.

\bibitem{Yehia}
S.~E.-B Yehia.
\newblock {\em Extensions of Simple Modules for the Universal Chevalley Groups
  and its Parabolic Subgroups}.
\newblock PhD thesis, Mathematics Institute, Warwick, 1982.

\end{thebibliography}

\end{document}